\documentclass[article,authoryear,equation]{beg_32}      
\def\jname{Preprint}\def\printNo{}

\usepackage{csl}

\usepackage{bbm,mathrsfs,mathtools,amsmath,amssymb,amsfonts,stmaryrd,upgreek}
\usepackage{graphicx}
\usepackage{enumitem,url,hyperref}

\usepackage{tikz}
\usepackage{tikz-qtree}
\usetikzlibrary{positioning}
\usetikzlibrary{calc}

\newif\ifcomplete

\providecommand{\half}{\scalebox{0.7}{$\displaystyle\frac{1}{2}$}}

\newcommand{\Hobs}{\mathcal{H}}

\renewcommand{\P}{\mathbf{P}}

\usepackage{accents}
\newcommand*{\bigdot}[1]{\accentset{\mbox{\large\bfseries .}}{#1}}

\newcommand{\x}[1][]{%
   \ifthenelse{ \equal{#1}{} }
      {X}
      {X^{[#1]}}}
      
\renewcommand{\u}{U}
\newcommand{\uh}[1][]{%
   \ifthenelse{ \equal{#1}{} }
      {\hat{U}}
      {\hat{u}^{(#1)}}}
\newcommand{\xb}[1][]{%
   \ifthenelse{ \equal{#1}{} }
      {X^{\mathrm{b}}}
      {\mathbf{X}^{\mathrm{b},[#1]}}}
      
\newcommand{\ub}[1][]{%
   \ifthenelse{ \equal{#1}{} }
      {U^{\mathrm{b}}}
      {\mathbf{U}^{\mathrm{b},[#1]}}}
      
\newcommand{\ua}[1][]{%
   \ifthenelse{ \equal{#1}{} }
      {U^{\mathrm{a}}}
      {\mathbf{U}^{\mathrm{a},[#1]}}}
      
\newcommand{\uha}[1][]{%
   \ifthenelse{ \equal{#1}{} }
      {\hat{U}^{\mathrm{a}}}
      {\mathbf{\hat{U}}^{\mathrm{a},[#1]}}}
      
\newcommand{\pub}[1][]{%
   \ifthenelse{ \equal{#1}{} }
      {\*\Phi_r\,U^{\mathrm{b}}}
      {\*\Phi_r\,\mathbf{U}^{\mathrm{b},[#1]}}}
      
\newcommand{\uhb}[1][]{%
   \ifthenelse{ \equal{#1}{} }
      {\hat{U}^{\mathrm{b}}}
      {\mathbf{\hat{U}}^{\mathrm{b},[#1]}}}
      
\newcommand{\puhb}[1][]{%
   \ifthenelse{ \equal{#1}{} }
      {\*\Phi_r\,\hat{U}^{\mathrm{b}}}
      {\*\Phi_r\,\mathbf{\hat{U}}^{\mathrm{b},#1}}}
      
\newcommand{\zb}[1][]{%
   \ifthenelse{ \equal{#1}{} }
      {Z^{\mathrm{b}}}
      {\mathbf{Z}^{\mathrm{b},#1}}}
      
\newcommand{\za}[1][]{%
   \ifthenelse{ \equal{#1}{} }
      {Z^{\mathrm{a}}}
      {\mathbf{Z}^{\mathrm{a},[#1]}}}

\newcommand{\xa}[1][]{%
   \ifthenelse{ \equal{#1}{} }
      {X^{\rm a}}
      {\mathbf{X}^{\mathrm{a},[#1]}}}
\newcommand{\xf}[1][]{%
   \ifthenelse{ \equal{#1}{} }
      {\mathbf{x}^{\rm f}}
      {\mathbf{x}^{{\rm f}[#1]}}}

\newcommand{\hofx}{\mathbf{z}}

\newcommand{\hofxb}[1][]{%
   \ifthenelse{ \equal{#1}{} }
      {\hofx^{\mathrm{b}}}
      {\hofx^{{\mathrm{b}}[#1]}}}
\newcommand{\hofxa}[1][]{%
   \ifthenelse{ \equal{#1}{} }
      {\hofx^{\rm a}}
      {\hofx^{{\rm a}[#1]}}}
\newcommand{\dhofxb}[1][]{%
   \ifthenelse{ \equal{#1}{} }
      {\bigdot{\hofx}^{\mathrm{b}}}
      {\bigdot{\hofx}^{{\mathrm{b}}[#1]}}}

\newcommand{\errb}[1][]{%
   \ifthenelse{ \equal{#1}{} }
      {\varepsilon^{\mathrm{b}}}
      {\varepsilon^{{\mathrm{b}}[#1]}}}
\newcommand{\erra}[1][]{%
   \ifthenelse{ \equal{#1}{} }
      {\varepsilon^{\rm a}}
      {\varepsilon^{{\rm a}[#1]}}}
\newcommand{\erro}[1][]{%
   \ifthenelse{ \equal{#1}{} }
      {\eta}
      {\boldsymbol{\eta}^{(#1)}}}

\newcommand{\errm}[1][]{%
   \ifthenelse{ \equal{#1}{} }
      {\eta}
      {\eta^{[#1]}}}
\newcommand{\wb}[1][]{%
   \ifthenelse{ \equal{#1}{} }
      {w^{\mathrm{b}}}
      {w^{{\mathrm{b}}[#1]}}}
\newcommand{\wa}[1][]{%
   \ifthenelse{ \equal{#1}{} }
      {w^{\rm a}}
      {w^{{\rm a}[#1]}}}

\usepackage{amsthm}

\graphicspath{ {Figures/} }

\def\!#1{\mathcal{#1}}
\def\*#1{\boldsymbol{\mathbf{#1}}}
\def\|#1{\textnormal{#1}}

\def\norm#1{\left\lVert#1\right\rVert}

\newcommand{\Mean}[1]{\*\mu_{#1}}
\newcommand{\MeanE}[1]{\widetilde{\*\mu}_{#1}}
\newcommand{\Cov}[2]{\*\Sigma_{#1,#2}}
\newcommand{\CovE}[2]{\widetilde{\*\Sigma}_{#1,#2}}

\newcommand{\En}[1]{\mathsf{E}_{#1}}

 \csltitle{Multifidelity Ensemble Kalman Filtering Using Surrogate Models Defined by Physics-Informed Autoencoders}
 \cslauthor{Andrey A. Popov, Adrian Sandu}
 \cslyear{21}
 \cslreportnumber{1}
 \cslemail{apopov@vt.edu, sandu@cs.vt.edu} 
 
\renewcommand\fottitle{}

\usepackage[hang]{footmisc}
\fancypagestyle{plain}{%
    \fancyhf{}
    \fancyhead[R]{\small {\it \jname}, x(x):\thepage--\pageref{LastPage} (\myyear\today)}
    \fancyfoot[R]{\small\bf\thepage }
    \fancyfoot[L]{\fottitle}
}
\renewcommand{\dmyy}{21}
\renewcommand{\myyear}{2021}
\renewcommand{\today}{}

\begin{document}
    \graphicspath{ {Figures/} }

        \csltitlepage
        
            \volume{Volume x, Issue x, \myyear\today}%
            \title{Multifidelity Ensemble Kalman Filtering Using Surrogate Models Defined by Physics-Informed Autoencoders}%
            \titlehead{Nonlinear MFEnKF}%
            \authorhead{A.A. Popov, \& A. Sandu}%
            \corrauthor[1]{Andrey A. Popov}%
            \author[1]{Adrian Sandu}%
            \corremail{apopov@vt.edu}%
            \address{2202 Kraft Drive, Blacksburg, VA, 24060}%
            \dataO{mm/dd/yyyy}%
            \dataF{mm/dd/yyyy}%
            \abstract{Data assimilation is a Bayesian inference process that obtains an enhanced understanding of a physical system of interest by fusing information from an inexact physics-based model, and from noisy sparse observations of reality. The multifidelity ensemble Kalman filter (MFEnKF) recently developed by the authors combines a full-order physical model and a hierarchy of reduced order surrogate models in order to increase the computational efficiency of data assimilation. The standard MFEnKF uses linear couplings between models, and is statistically optimal in case of Gaussian probability densities. This work extends MFEnKF to work with non-linear couplings between the models. Optimal nonlinear projection and interpolation operators are obtained by appropriately trained physics-informed autoencoders, and this approach allows to construct reduced order surrogate models with less error than conventional linear methods. Numerical experiments with the canonical Lorenz '96 model illustrate that nonlinear surrogates perform better than linear projection-based ones in the context of multifidelity filtering.
            }%
            \maketitle%

\tableofcontents




\newpage
\setcounter{page}{1}

\section{Introduction}

Data assimilation~\citep{reich2015probabilistic,asch2016data} is a Bayesian inference process that fuses information obtained from an inexact physics-based model, and from noisy sparse observations of reality, in order to enhance our understanding of a physical process of interest. The reliance on physics-based models distinguishes data assimilation from traditional machine learning methodologies, which aim to learn the quantities of interest through purely data-based approaches.  From the perspective of machine learning, data assimilation is a learning problem where the quantity of interest is constrained by prior physical assumptions, as captured by the model, and nudged towards the optimum solution by small amounts of data from imperfect observations. Therefore data assimilation can be considered a form of physics-constrained machine learning  \citep{karpatne2017theory,willard2020integrating}. This work improves data assimilation methodologies by combining a mathematically rigorous data assimilation approach and a data rigorous machine learning algorithm through powerful techniques in multilevel inference \citep{giles2008multilevel,giles2015multilevel}. 

The ensemble Kalman filter~\citep{evensen1994sequential,evensen2009data,Burgers_1998_EnKF} (EnKF) is a family of computational methods that tackle the data assimilation problem using Gaussian assumptions, and a Monte Carlo approach where the underlying probability densities are represented by ensemble of model state realizations.
%
%
The ensemble size, i.e., the number of physics-based model runs, is typically the main factor that limits the efficiency of EnKF. For increasing the quality of the results when ensembles are small, heuristics correction methods such as covariance shrinkage~\citep{Sandu_2019_Covariance-parallel,Sandu_2015_covarianceShrinkage,Sandu_2018_Covariance-Cholesky} and localization~\citep{petrie2008localization,Sandu_2019_adaptive-localization,Sandu_2019_ML-localization} have been developed. As some form of heuristic correction is required for operation implementations of the ensemble Kalman filter, reducing the need for such heuristic corrections in operational implementations is an important and active area of research.


The dominant cost in operational implementations of EnKF is the large number of expensive high fidelity physics-based model runs, which we refer to as ``full order models'' (FOM). A natural approach to  increase efficiency is to endow the data assimilation algorithms with the ability to make use of inexpensive but less accurate model runs \citep{cao2007reduced,farrell2001state}, which we refer to as ``reduced order models'' (ROMs).  ROMs are constructed to capture the most important aspects of the dynamics of the FOM, at a fraction of the computational cost; typically they use a much smaller number of variables than the corresponding FOM.
The idea of leveraging model hierarchies in numerical algorithms for uncertainty quantification \citep{peherstorfer2018survey} is fast gaining traction in both the data assimilation and machine learning communities. Here we focus on two particular types of ROMs: a proper orthogonal decomposition (POD) based ROM, corresponding to a linear projection of the FOM dynamics onto a small linear subspace \citep{brunton2019data}, and a ROM based on autoencoders  \citep{aggarwal2018neural}, which are neural networks built for optimal dimensionality reduction, corresponding to a non-linear projection of the dynamics onto a small dimensional manifold.

The multifidelity ensemble Kalman filter (MFEnKF) \citep{popov2021multifidelity,popov2021chapter} is a correction and extension to the multilevel EnKF \citep{Hoel_2016_MLEnKF,Chernov_2021_MLEnKF,chada2020multilevel,hoel2019multilevel} that combines the ensemble Kalman filter with the idea of surrogate modeling. The MFEnKF attempts to utilize both the full-order and a reduced order surrogate model in order to optimally combine the information obtained from model runs with information begotten from the observations. By posing the problem in terms of a mathematically rigorous variance reduction technique---the linear control variate framework---MFEnKF is able to provide robust guarantees about the accuracy of the inference results.

While numerical weather prediction is the dominant driver of innovation in data assimilation literature~\citep{kalnay2003atmospheric}, other applications can benefit from our multifidelity approach such as mechanical engineering~\citep{Sandu_2009_ParameterExplicit,Sandu_2010_paramEstimation,Sandu_2010_polychaosEKF} and air quality modeling~\citep{Sandu_2007_EnKF_localization,Sandu_2007_EnKFgeneral,Sandu_2007_EnKFideal}.

This paper is organized as follows. Section \ref{sec:background} discusses the data assimilation problem, provides background on control variates, the EnKF, and the MFEnKF, as well as ROMs and autoencoders. Section \ref{sec:NL-MFEnKF} introduces the non-linear extension to the MFEnKF. Section \ref{sec:models} presents the Lorenz '96 model and the corresponding POD-ROM and NN-ROM surrogates. Section \ref{sec:numerical-experiments} provides the results of numerical experiments. Concluding remarks are made in Section \ref{sec:conclusions}.

\section{Background}
\label{sec:background}

 As the aphorism goes ``all models are wrong, but some are useful'', Sequential data assimilation propagates our imperfect knowledge about some physical quantity of interest through an imperfect model which represents a time-evolving dynamical system, typically a chaotic system \citep{strogatz2018nonlinear}. Without an additional influx of external information, our knowledge about the systems would rapidly degrade, therefore acquiring and assimilating imperfect external information about the quantity of interest is the core of the problem that we are interested in.

To be specific, consider a dynamical system of interest whose true state at time $t_i$ is $X^t_i$. The evoluton of the system is captured by the dynamical model
\begin{equation}\label{eq:model}
    X_{i} = \!M_{t_{i-1}, t_i}(X_{i-1}) + \Xi_i, 
\end{equation}
where $X_i$ is a random variable whose distribution describes our knowledge of the the state of a physical process at time index $i$, and $\Xi_i$ is a  random variable describing the modeling error. In this paper we assume a perfect model ($\Xi_i \equiv \*0$), as the discussion of model error in multifidelity methods is significantly outside the scope of this paper.

Additional independent information about the system is obtained through imperfect measurements of the observable aspects $Y_i$ of the truth $X^t_i$, i.e., through noisy observations 
\begin{equation}\label{eq:observation}
    Y_i = \Hobs(X^t_i) + \eta_i, \quad \eta_i\sim\!N(\*0, \Cov{\eta_i}{\eta_i}),
\end{equation}
where the ``observation operator'' $ \Hobs$ maps the model space onto the observation space (i.e., selects the observable aspects of the state). 

Our aim is to combine the two sources of information in a Bayesian framework:
\begin{equation}
    \pi(X_i | Y_i) \propto \pi(Y_i | X_i)\,\pi(X_i).
\end{equation} 

In the remainder of the paper we use the following notation. Let $W$ and $V$ be  random variables.  The exact mean of $W$ is denoted by $\Mean{W}$, and the exact covariance between $W$ and $V$ by $\Cov{W}{V}$. $\En{W}$ denotes an ensemble of samples of $W$, and $\MeanE{W}$ and $\CovE{W}{V}$ are the empirical ensemble mean of $W$ and empirical ensemble covariance of $W$ and $V$, respectively.

\subsection{Control variates}\label{sec:cv}

Bayesian inference requires that all available information is used in order to obtain correct results \citep{jaynes2003probability}. Variance reduction techniques \citep{mcbook} are methods that provide estimates of some quantity of interest with lower variability. From a Bayesian perspective they represent a reintroduction of information that was previously ignored. The linear control variate (LCV) approach \citep{rubinstein1985efficiency} is a variance reduction method that aims to incorporate the new information in an optimal linear sense.

LCV works with two vector spaces, a principal space $\mathbb{X}$ and a control space $\mathbb{U}$, and several random variables, as follows. The principal variate $X$ is a $\mathbb{X}$-valued random variable, and the control variate $\hat{U}$ is  a highly correlated $\mathbb{U}$-valued random variable. The ancillary variate $U$ is a $\mathbb{U}$-valued random variable, which is uncorrelated with the preceding two variables, but shares the same mean as $\hat{U}$, meaning $\Mean{U} = \Mean{\hat{U}}$.
The linear control variate framework builds a new $\mathbb{X}$-valued random variable  $Z$, called the total variate:
\begin{equation}
\label{eq:control-variate}
    Z = X - \*S(\hat{U} - U),
\end{equation}
where the  linear ``gain'' operator $\*S$ is used to minimize the variance in $Z$ by utilizing the information about the distributions of the constituent variates $X$, $\hat{U}$ and $U$.

In this work $\mathbb{X}$ and $\mathbb{U}$ are finite dimensional vector spaces. The dimension of $\mathbb{X}$ is $n$, the dimension of $\mathbb{U}$ is denoted by $m$ when it is the observation space, and by $r$ when it is the reduced order model state space.

The following lemma is a slight generalization of \cite[Appendix]{rubinstein1985efficiency}.
\begin{lemma}[Optimal gain]
\label{lem:optimal-gain}
The optimal gain $\*S$ that minimizes the trace of the covariance of $Z$ \eqref{eq:control-variate} is:
\begin{equation}
    \*S = \Cov{X}{\hat{U}}{(\Cov{\hat{U}}{\hat{U}} + \Cov{U}{U})}^{-1}.
\end{equation}
\end{lemma}

Here we focus on the case where the principal and control variates are related through a deterministic projection operator:
\begin{equation}\label{eq:projection}
    \hat{U} = \theta(X),
\end{equation}
and where there exists an interpolation operator such that:
\begin{equation}\label{eq:interpolation}
    X \approx \widetilde{X} = \phi(\hat{U}),
\end{equation}
which in some optimal sense approximately recovers the information embedded in $X$. The variable $\widetilde{X}$ is known as the reconstruction.

We make the assumption that the projection and interpolation operators are linear:
\begin{equation}
\label{eq:linear-projection-interpolation}
\begin{gathered}
    \theta(X) \coloneqq \*\Theta\, X,\qquad
    \phi(\hat{U}) \coloneqq \*\Phi\,\hat{U}.
    \end{gathered}
\end{equation}
We reproduce below the useful result \cite[Theorem 3.1]{popov2021multifidelity}.
\begin{theorem}
\label{thm:optimal-S}
Under the assumptions that $\hat{U}$ and $U$ have equal covariances, and that the principal variate residual is uncorrelated with the control variate, the optimal gain of \eqref{eq:control-variate} is half the interpolation operator:
\begin{equation}
\Cov{\hat{U}}{\hat{U}} = \Cov{U}{U} \quad \textnormal{and} \quad
    \Cov{(X - \*\Phi\hat{U})}{\hat{U}} = \*0
    \quad \Rightarrow \quad \*S = \half\, \*\Phi.
\end{equation}
\end{theorem}

Under the assumptions of Theorem \ref{thm:optimal-S} the control variate structure \eqref{eq:control-variate} is:
\begin{equation}\label{eq:optimal-control-variate}
    Z = X - \half\, \,\*\Phi\,(\hat{U} - U).
\end{equation}
    
\subsection{Linear control variates with non-linear transformations}
While working with linear transformations is easy, many applications require reducing the variance of a non-linear transformation of a random variable. We generalize the control variate framework to such an application. 

Following \cite{mcbook}, assume that our transformed principal variate is of the form $h(X)$, and the transformed control variate is of the form $g(\hat{U})$, where $h$ and $g$ are smooth non-linear functions.  Similarly, assume that the transformed ancillary variate has the form $g(U)$. The total variate $Z^h$, in the same space as $h(X)$, is defined as: 
\begin{equation}
\label{eq:nonlinear-control-variate}
    Z^h = h(X) - \*S\left( g(\hat{U}) - g(U) \right),
\end{equation}
with the optimal linear gain given by Lemma \ref{lem:optimal-gain}:
\begin{equation}
    \*S = \Cov{h(X)}{g(\hat{U})}\left(\Cov{g(\hat{U})}{g(\hat{U})} + \Cov{g(U)}{g(U)}\right)^{-1}.
\end{equation}
%
%
\begin{theorem}\label{thm:non-linear-control-gaussian}
If $\hat{U}$ and $U$ are Gaussian, and share the same mean and covariance, ($\Mean{\hat{U}} = \Mean{U}$, $\Cov{\hat{U}}{\hat{U}} = \Cov{U}{U}$), then 
\begin{equation}
        \Mean{g(\hat{U})} = \Mean{g(U)}, \quad\textnormal{and}\quad \Cov{g(\hat{U})}{g(\hat{U})} = \Cov{g(U)}{g(U)},
\end{equation}
the control variate structure \eqref{eq:nonlinear-control-variate} holds, and the optimal linear gain is:
\begin{equation}
        \*S = \half\, \Cov{h(X)}{g(\hat{U})}\Cov{g(\hat{U})}{g(\hat{U})}^{-1}.
\end{equation}
\end{theorem}
\begin{proof}
    As $\hat{U}$ and $U$ are Gaussian, they are completely defined by their mean and covariance, and any deterministic non-linear transformation of the variates has the same moments, thereby sharing the mean and covariance. In that case $g(\hat{U}) - g(U)$ is unbiased, and the control variate framework can be applied.
\end{proof}

We now consider the case where the principal variate transformation is linear, $h = \*H$, and the control and ancillary variate transformation is the interpolation operator, $g = \phi$, such that the underlying control variate is a projection of the principal variate \eqref{eq:projection},
\begin{equation}
    \*H\, X \approx \phi(\hat{U}).
\end{equation}

\begin{theorem}\label{thm:non-linear-optimal-gain}
    Without proof, if the assumption of Theorem \ref{thm:non-linear-control-gaussian} hold, and
    the transformed principal variate residual is uncorrelated with the transformed control variate,
\begin{equation}
        \Cov{\left(\*H X - \phi(\hat{U})\right)}{\phi(\hat{U})} \stackrel{\rm assumed}{=} \*0,
\end{equation}
    then the optimal gain is 
\begin{equation}
        \*S = \half\, \*I,
\end{equation}
    where $\*I$ is the identity operator.
\end{theorem}

\subsection{The Ensemble Kalman Filter}

The EnKF is a statistical approximation to the optimal control variate structure \eqref{eq:control-variate}, where the underlying probability density functions are represented by empirical measures using ensembles, i.e., a finite number of realizations of the random variables. The linear control variate framework allows to combine multiple ensembles into one that better represents the desired quantity of interest.

Let $\En{X^b_i} \in \mathbb{R}^{n\times N_X}$ be an ensemble of $N_X$ realizations of the $n$-dimensional principal variate, which represents our prior uncertainty in the model state at time index $i$ from \eqref{eq:model}. Likewise, let $\En{\Hobs_i(X^b_i)} = \Hobs_i(\En{X^b_i}) \in \mathbb{R}^{m\times N_X}$ be an ensemble of $N_X$ realizations of the $m$-dimensional control observation state variate, which represents the same model realizations cast into observation space. Let $\En{Y_i} \in \mathbb{R}^{m\times N_X}$ be an ensemble of $N_X$ `perturbed observations', which is a statistical correction required in the ensemble Kalman filter \citep{Burgers_1998_EnKF}.

\begin{remark}[EnKF Perturbed Observations]
\label{rem:perturbed-observations}
Each ensemble member of the perturbed observations is sampled from a Gaussian distribution with mean the measured value, and the known observation covariance from \eqref{eq:observation}:
\begin{equation}\label{eq:EnKF-perturbed-observations}
    [\En{Y_i}]_{:,e} \sim \!N(\Mean{Y_i}, \Cov{\eta_i}{\eta_i}). 
\end{equation}
\end{remark}

The prior ensemble at time step $i$ is obtained by propagating the posterior ensemble at time $i-1$ through the model equations,
\begin{equation}
    \En{X^b_i} = \!M_{t_{i-1},t_i}\left(\En{X^a_{i-1}}\right),
\end{equation}
where the slight abuse of notation indicates an independent model propagation of each ensemble member.
Application of the Kalman filter formula constructs an ensemble $\En{X^a_i}$ describing the posterior uncertainty:
\begin{equation}
    \En{X^a_i} = \En{X^b_i} - \widetilde{\*K}_i\left(\En{\Hobs_i(X^b_i)} - \En{Y_i}\right),
\end{equation}
where the  statistical Kalman gain is an ensemble-based approximation to the exact gain in Lemma \ref{lem:optimal-gain}:
\begin{equation}
    \widetilde{\*K}_i = \CovE{X^b_i}{\Hobs_i(X^b_i)}\left(\CovE{\Hobs_i(X^b_i)}{\Hobs_i(X^b_i)} + \Cov{\eta_i}{\eta_i}\right)^{-1}.
\end{equation}

\begin{remark}[Inflation]\label{rem:inflation}
Inflation is a probabilistic correction necessary to account for the Kalman gain being correlated to the ensemble  \citep{popov2020explicit}. In inflation the ensemble anomalies (deviations from the statistical mean) are multiplied by a constant $\alpha>1$, thereby increasing the covariance of the distribution described by the ensemble:
\begin{equation}
        \En{\xb_{i+1}} \leftarrow \MeanE{\xb_{i+1}} + \alpha \left(\En{\xb_{i+1}} - \MeanE{\xb_{i+1}}\right).
\end{equation}
\end{remark}

\subsection{The Multifidelity Ensemble Kalman Filter}

The Multifidelity Ensemble Kalman Filter (MFEnKF) \citep{popov2021multifidelity} seeks to merge the information from a hierarchy of models and the corresponding observations into a coherent representation that can be propagated forward in time. This necessitates that the information from the  models is decoupled, but implicitly preserving of some underlying structure; we take the linear control variate structure for convenience.

Without loss of generality, we discuss here a bi-fidelity approach, with a telescopic extension to multiple fidelities provided at the end of the section.
Instead of having access to one model $\!M$, assume that we have access to a hierarchy of models. In the bi-fidelity case, the principal space model is denoted by $\!M^X$ and the control space model is denoted by $\!M^U$. 
%

We now consider the total variate
\begin{equation}\label{eq:MFEnKF-linear-control-variate}
    \zb_i = \xb_i - \half\, \*\Phi\,(\uhb_i - \ub_i),
\end{equation}
that describes the prior total information from a model that evolves in principal space ($\xb_i$) and a model that evolves in ancillary space ($\uhb_i$ and $\ub_i$).

Assume that our prior total variate is represented by the three ensembles $\En{\xb_i} \in \mathbb{R}^{n\times N_X}$ consisting of $N_X$ realizations of the $n$-dimensional principal model state variate , $\En{\uhb_i} \in \mathbb{R}^{r\times N_X}$ consisting of $N_X$ realizations of the $r$-dimensional control model state variate, and $\En{\ub_i} \in \mathbb{R}^{r\times N_U}$ consisting of $N_U$ realizations of the $r$-dimensional ancillary model state variate. Each of these ensembles has a corresponding ensemble of $m$-dimensional control observation space realizations.

MFEnKF performs sequential data assimilation using the above constituent ensembles, without having to explicitly calculate the ensemble of the total variates. The MFEnKF forecast step propagates the three ansembles form the previous step:
\begin{equation}
    \begin{aligned}
    \En{\xb_i} &= \!M^X_{t_{i-1},t_i}(\En{\xa_{i-1}}),\\
    \En{\uhb_i} &= \!M^U_{t_{i-1},t_i}(\En{\uha_{i-1}}),\\
    \En{\ub_i} &= \!M^U_{t_{i-1},t_i}(\En{\ua_{i-1}}).
    \end{aligned}
\end{equation}

Two observation operators $\Hobs_i^X$ and $\Hobs_i^U$ cast the principal model and control model spaces, respectively, into the control observation space. In this paper we assume that he principal model space observation operator is the canonical observation operator \eqref{eq:observation}:
\begin{equation}\label{eq:MFEnKF-principal-model-space-observation}
    \Hobs^X_i(X_i) \coloneqq \Hobs_i(X_i),
\end{equation}
and that the control model space observation operator is the canonical observation operator \eqref{eq:observation} applied to the linear interpolated reconstruction \eqref{eq:interpolation} of a variable in control model space:
\begin{equation}\label{eq:MFEnKF-control-model-space-observation}
    \Hobs^U_i(U_i) \coloneqq \Hobs_i(\*\Phi U_i).
\end{equation}
Additionally, we define an (approximate) observation operator for the total model variate :
\begin{equation}\label{eq:MFEnKF-total-model-space-observation}
    \Hobs_i^Z(Z_i) \coloneqq \Hobs_i^X(X_i) - \half\, \left(\Hobs_i^U(\hat{U}_i) - \Hobs_i^U(U_i)\right),
\end{equation}
which, under the linearity assumptions on $\Hobs_i^X$ of Theorem \ref{thm:non-linear-optimal-gain} and the underlying Gaussian assumptions on $\hat{U}_i$ and $U_i$ of Theorem \ref{thm:non-linear-control-gaussian}, begets that $\Hobs_i^Z = \Hobs_i^X$. Even without the linearity assumption  the definition \eqref{eq:MFEnKF-total-model-space-observation} is operationally useful.

The MFEnKF analysis updates each constituent ensemble as follows:
\begin{equation}\label{eq:MFEnKF}
    \begin{aligned}
    \En{\xa_i} &= \En{\xb_i} - \widetilde{\mathbf{K}}_i\,\left(\En{\Hobs_i^X(\xb_i)} - \En{Y^{\x}_i}\right),\\
     \En{\uha_i} &= \En{\uhb_i} - \*\Theta\widetilde{\mathbf{K}}_i\,\left((\En{\Hobs_i^U(\uhb_i)} - \En{Y^{\x}_i}\right),\\
     \En{\ua_i} &= \En{\ub_i} - \*\Theta\widetilde{\mathbf{K}}_i\,\left((\En{\Hobs_i^U(\ub_i)} - \En{Y^{\u}_i}\right),
    \end{aligned}
\end{equation}
with the heuristic correction to the means
\begin{equation}\label{eq:MFEnKF-mean-correction}
    \MeanE{\xa_i} \leftarrow \MeanE{Z_i^a},\quad \MeanE{\uha_i} \leftarrow \*\Theta\MeanE{Z_i^a},\quad \MeanE{\ua_i} \leftarrow \*\Theta\MeanE{Z_i^a},
\end{equation}
applied in order to fulfill the unbiasedness requirement of the control variate structure:
\begin{equation}\label{eq:MFEnKF-mean}
\MeanE{\za_i} = \MeanE{\zb_i} - \widetilde{\mathbf{K}}_i\,\left(\MeanE{\Hobs_i^X(\zb_i)} - \Mean{Y_i}\right).
\end{equation}
The Kalman gain and the covariances are defined by the semi--linearization:
\begin{align}
    \widetilde{\mathbf{K}}_i &= \CovE{\zb_i}{\Hobs_i^Z(\zb_i)}\left(\CovE{\Hobs_i^Z(\zb_i)}{\Hobs_i^Z(\zb_i)} + \Cov{\eta_i}{\eta_i}\right)^{-1}\\
    \CovE{\zb_i}{\Hobs_i^Z(\zb_i)} &=
    \begin{multlined}[t]
        \CovE{\xb_i}{\Hobs_i^X(\xb_i)} + \frac{1}{4}\CovE{\*\Phi\uhb_i}{\Hobs_i^U(\uhb_i)} + \frac{1}{4}\CovE{\*\Phi\ub_i}{\Hobs_i^U(\ub_i)}\\
        -\half\, \CovE{\xb_i}{\Hobs_i^U(\uhb_i)} - \half\, \CovE{\*\Phi\uhb_i}{\Hobs_i^X(\xb_i)},
    \end{multlined} \label{eq:MFEnKF-Cov1}\\
    \CovE{\Hobs_i^Z(\zb_i)}{\Hobs_i^Z(\zb_i)} &=
    \begin{multlined}[t]
        \CovE{\Hobs_i^X(\xb_i)}{\Hobs_i(\xb_i)} + \frac{1}{4}\CovE{\Hobs_i^U(\uhb_i)}{\Hobs_i^U(\uhb_i)} + \frac{1}{4}\CovE{\Hobs_i^U(\ub_i)}{\Hobs_i^U(\ub_i)}\\
        -\half\, \CovE{\Hobs_i^X(\xb_i)}{\Hobs_i^U(\uhb_i)} - \half\, \CovE{\Hobs_i^U(\uhb_i)}{\Hobs_i^X(\xb_i)}.
    \end{multlined}\label{eq:MFEnKF-Cov2}
\end{align}

In order to ensure that the control variate $\hat{U}$ remains highly correlated to the principal variate $X$, at the end of each analysis step we replace the analysis control variate ensemble with the corresponding projection of the principal variate ensemble:
\begin{equation}\label{eq:MFEnKF-control-variate-correction}
    \En{\uha_i} \leftarrow \*\Theta\,\En{\xa_i}.
\end{equation}

\begin{remark}[MFEnKF Perturbed observations]\label{rem:MFEnKF-perturbed-observations}
    There is no unique way to perform perturbed observations (remark \ref{rem:perturbed-observations}) in the MFEnKF. We will present one way in this paper. As Theorem \ref{thm:non-linear-control-gaussian} requires both the control and ancillary variates to share the same covariance, we utilize here the `control space uncertainty consistency' approach.  The perturbed observations ensembles in \eqref{eq:MFEnKF} is defined by:
    \begin{align}
        [\En{Y^{\x}_i}]_{:,e} &\sim \!N(\Mean{Y_i}, \Cov{\eta_i}{\eta_i}),\\
        [\En{Y^{\u}_i}]_{:,e} &\sim \!N(\Mean{Y_i}, s\Cov{\eta_i}{\eta_i}),
    \end{align}
    where the scaling factor is $s = 1$.
    See  \cite[Section 4.2]{popov2021multifidelity}  for a more detailed discussion about perturbed observations.
\end{remark}

\begin{remark}[MFEnKF Inflation]
    Similarly to the EnKF (see Remark \ref{rem:inflation}), the MFEnKF also requires inflation in order to account for the statistical Kalman gain being correlated to its constituent ensembles. For a bi-fidelity MFEnKF, two inflation factors are required: $\alpha_X$ which acts on the anomalies of the principal and control variates (as they must remain highly correlated) and $\alpha_U$ which acts on the ensemble anomalies of the ancillary variate:
\begin{equation}
        \begin{gathered}
        \En{\xb_{i+1}}  \leftarrow \MeanE{\xb_{i+1}} + \alpha_X \left(\En{\xb_{i+1}} - \MeanE{\xb_{i+1}}\right),\\ \En{\uhb_{i+1}} \leftarrow \MeanE{\uhb_{i+1}} + \alpha_X \left(\En{\uhb_{i+1}} - \MeanE{\uhb_{i+1}}\right),\\
        \En{\ub_{i+1}} \leftarrow \MeanE{\ub_{i+1}} + \alpha_U \left(\En{\ub_{i+1}} - \MeanE{\ub_{i+1}}\right).
        \end{gathered}
\end{equation}
\end{remark}

\subsection{Projection-based reduced order models} 
It is well known that the important information of many dynamical systems can be expressed with significantly fewer dimensions than the discretization dimension $n$ \citep{sell2013dynamics}. For many infinite dimensional equations it is possible to construct a finite-dimensional inertial manifold that represents the dynamics of the system (including the global attractor). The Hausdorff dimension of the global attractor of some dynamical system is a useful lower bound for the minimal representation of the dynamics, though a representation of just the attractor is likely not sufficient to fully capture all the `useful' aspects of the data.
For data-based reduced order models an important aspect is the intrinsic dimension \citep{lee2007nonlinear} of the data. The authors' are not aware of any formal statements relating the dimension of an inertial manifold and the intrinsic dimension of some finite discretization of the dynamics.
For some reduced dimension $r$ that is sufficient to represent either the dynamics or the data, or both, it is possible to build a `useful' surrogate model.

We will now discuss the construction of reduced order models for problems posed as ordinary differential equations.
The following derivations are similar to those found in \cite{lee2020model}, but assume vector spaces and no re-centering.

Just like in the control variate framework in section \ref{sec:cv}, our full order model will reside in the principal space $\mathbb{X}$ and our reduced order model will be defined in the space  $\mathbb{U}$, which is related to  $\mathbb{X}$ through some coupling.

Given an initial value problem in $\mathbb{X}$:
\begin{equation}
    \frac{\mathrm{d} X}{\mathrm{d} t} = f(X), \quad X(t_0) = X_0, \quad t \in [t_0, t_f],
\end{equation}
and the projection operator \eqref{eq:projection}, the induced reduced order model initial value problem in $\mathbb{U}$ is defined by simple differentiation of $U = \theta(X)$, by dynamics in the space $\mathbb{U}$,
\begin{equation}
    \frac{\mathrm{d} U}{\mathrm{d} t} = \theta'(X)f(X), \quad X(t_0) = X_0, \quad t \in [t_0, t_f].
\end{equation}
As is common, the full order trajectory is not available during integration, as there is no bijection from $\mathbb{X}$ to $\mathbb{U}$, thus an approximation using the interpolation operator \eqref{eq:interpolation} that fully resides in $\mathbb{U}$ is used instead:
\begin{equation}
\label{eq:reduced-order-model}
    \frac{\mathrm{d} U}{\mathrm{d} t} = \theta'(\phi(U))f(\phi(U)), \quad U(t_0) = \phi(X_0), \quad t \in [t_0, t_f].
\end{equation}

Note that this is not the only way to obtain a reduced order model by using arbitrary projection and interpolation operators. It is however the simplest extension of the POD-based ROM framework.

\begin{remark}[Linear ROM]
In the linear case \eqref{eq:linear-projection-interpolation} the reduced order model \eqref{eq:reduced-order-model} takes the form
\begin{equation}\label{eq:linear-reduced-order-model}
    \frac{\mathrm{d} U}{\mathrm{d} t} = \*\Theta f(\*\Phi U), \quad U(t_0) = \phi(X_0), \quad t \in [t_0, t_f].
\end{equation}
\end{remark}

\subsection{Autoencoders}

An autoencoder~\citep{aggarwal2018neural} is an artificial neural network consisting of two components, an encoder $\theta$ and a decoder $\phi$, such that given a variable $X$ in the principal space, the variable
\begin{equation}
\label{eq:projection}
    \hat{U} = \theta(X),
\end{equation}
resides in the control space of the encoder. Conversely the reconstruction,
\begin{equation}
\label{eq:reconstruction}
    \tilde{X} = \phi(\hat{U}),
\end{equation}
is an approximation to $X$ in the principal space. In our context we think of the encoder as a nonlinear projection operator \eqref{eq:projection}, and of the decoder as a nonlinear interpolation operator \eqref{eq:interpolation}. 
While the relative dimension $n$ of the principal space is relatively high, the arbitrary structure of an artificial neural networks allows the autoencoder to learn the optimal $r$-dimensional (small) representation of the data.

\section{Non-linear projection-based MFEnKF}
\label{sec:NL-MFEnKF}

We extend MFEnKF to work with non-linear projection and interpolation operators. 
Theoretical extensions of the linear control variate framework to the non-linear case~\citep{nelson1987control} are not as-of-yet satisfactory. As is traditional in data assimilation literature, we resort to heuristic assumptions, which in practice may work significantly better than the theory predicts. The new algorithm is named NL-MFEnKF.

The main idea is to replace the optimal control variate structure for linear projection and interpolation operators \eqref{eq:optimal-control-variate} with one that works with their non-linear counterparts \eqref{eq:nonlinear-control-variate}:
\begin{equation}\label{eq:NL-MFEnKF-control-variate}
    Z_i^b = X_i^b - \half\, \left(\phi(\hat{U}_i^b) - \phi(U_i^b)\right).
\end{equation}
We assume that $\hat{U}$ and $U$ are Gaussian, and have the same mean and covariance, such that they obey the assumptions made in Theorem \ref{thm:non-linear-control-gaussian} and in Theorem \ref{thm:non-linear-optimal-gain} for the optimal gain.

 Similar to MFEnKF \eqref{eq:MFEnKF-control-model-space-observation}, the control model space observation operator  is the application of the canonical observation operator \eqref{eq:observation} to the reconstruction 
\begin{equation}\label{eq:NLMFEnKF-non-control-model-space-observation}
        \Hobs_i^U(U_i) \coloneqq \Hobs_i(\phi(U_i)),
\end{equation}
with the other observation operators (\ref{eq:MFEnKF-principal-model-space-observation},~\ref{eq:MFEnKF-total-model-space-observation}) defined as in the MFEnKF.
    
    \begin{remark}
        It is of independent interest to explore control model space observation operators that are not of the form \eqref{eq:NLMFEnKF-non-control-model-space-observation}. For example, if the interpolation operator $\phi$ is created through an autoencoder, the control model space observation operator $\Hobs^U$ could similarly be a different decoder of the same latent space.
    \end{remark}

The MFEnKF equations \eqref{eq:MFEnKF} are replaced by their non-linear counterparts in a manner similar to what is done with non-linear observation operators,
\begin{equation}\label{eq:NL-MFEnKF}
    \begin{aligned}
        \En{\xa_i} &= \En{\xb_i} - \widetilde{\mathbf{K}}_i\,\left(\En{\Hobs_i^X(\xb_i)} - \En{Y^{\x}_i}\right),\\
        \En{\uha_i} &= \En{\uhb_i} - \widetilde{\mathbf{K}}^\theta_i\,\left(\En{\Hobs_i^U(\uhb_i)} - \En{Y^{\x}_i}\right),\\
        \En{\ua_i} &= \En{\ub_i} - \widetilde{\mathbf{K}}^\theta_i\,\left(\En{\Hobs_i^U(\ub_i)} - \En{Y^{\u}_i}\right),\\
    \end{aligned}
\end{equation}
where, as opposed to \eqref{eq:MFEnKF}, there are now two Kalman gains, defined by,
\begin{align}
    \widetilde{\*K}_i &= \CovE{\zb_i}{\Hobs_i^Z(\zb_i)}\left(\CovE{\Hobs_i^Z(\zb_i)}{\Hobs_i^Z(\zb_i)} + \Cov{\eta_i}{\eta_i}\right)^{-1},\\
    \widetilde{\*K}^\theta_i &= \CovE{\theta(\zb_i)}{\Hobs_i^Z(\zb_i)}\left(\CovE{\Hobs_i^Z(\zb_i)}{\Hobs_i^Z(\zb_i)} + \Cov{\eta_i}{\eta_i}\right)^{-1},
\end{align}
for our purposes the covariances are semi-linear approximations, similar to \eqref{eq:MFEnKF-Cov1} and \eqref{eq:MFEnKF-Cov2},
and the perturbed observations are defined just like in the MFEnKF (Remark \ref{rem:MFEnKF-perturbed-observations}).
    
\subsection{NL-MFEnKF Heuristic Corrections}

For linear operators the projection of the mean is the mean of the projection. This is however not true for general non-linear operators. Thus in order to correct the means like in the MFEnKF \eqref{eq:MFEnKF-mean-correction}, additional assumptions have to be made.

The empirical mean of the total analysis variate (similar to \eqref{eq:NL-MFEnKF-control-variate} and \eqref{eq:MFEnKF-mean})  is
\begin{equation}\label{eq:NL-MFEnKF-mean-Z}
        \MeanE{Z_i^a} = \MeanE{X_i^a} - \half\, \left(\MeanE{\phi(\hat{U}_i^a)} - \MeanE{\phi(U_i^a)}\right).
\end{equation}
We use it to find the optimal mean adjustments in reduced space.
Specifically, we set the mean of the principal variate to be the mean of the total variate \eqref{eq:NL-MFEnKF-mean-Z},
\begin{equation}
    \MeanE{X_i^a} \leftarrow \MeanE{Z_i^a},
\end{equation}
enforce the recorrelation of the principal and control variates \eqref{eq:MFEnKF-control-variate-correction}) via
\begin{equation}
    \En{\hat{U}_i^a} \leftarrow \theta(\En{X_i^a}),
\end{equation}
and define the control variate mean adjustment as,
\begin{equation}\label{eq:NL-MFEnKF-control-variate-mean-correction}
    \MeanE{\hat{U}_i^a} \leftarrow \MeanE{\theta(X_i^a)}.
\end{equation}
Unlike the linear control variate framework of the MFEnKF \eqref{eq:MFEnKF-linear-control-variate}, the non-linear framework of the NL-MFEnKF \eqref{eq:NL-MFEnKF-control-variate} does not induce a unique way to impose unbiasedness on the control-space variates.
There are many possible non-linear formulations to the MFEnKF, and many possible heuristic corrections of the mean the ancillary variate. 
Here we discuss three approaches based on:
\begin{enumerate}[topsep=3pt,itemsep=0pt]
    \item control space unbiased mean adjustment,
    \item principal space unbiased mean adjustment, and
    \item Kalman approximate mean adjustment,
\end{enumerate}
each stemming from a different assumption on the relationship between the ancillary variate and the other variates.

\subsubsection{Control space unbiased mean adjustment}
By assuming that the control variate $\hat{U}_i^a$ and  and the ancillary variate $U_i^a$ are unbiased in the control space, the mean adjustment of $\hat{U}^a$ in \eqref{eq:NL-MFEnKF-control-variate-mean-correction} directly defines the mean adjustment of the ancillary variate:
\begin{equation}\label{eq:NL-MFEnKF-control-space-mean-correction}
    \MeanE{U_i^a} \leftarrow \MeanE{\hat{U}_i^a}.
\end{equation}
The authors will choose this method of correction in the numerical experiments for both its properties and ease of implementation.

\subsubsection{Principal space unbiased mean adjustment}
If instead we assume that the control variate $\phi(\hat{U}_i)$ and the ancillary variate $\phi(U_i)$ are unbiased in the principal space, then the mean of the total variate $Z_i$ \eqref{eq:NL-MFEnKF-control-variate} equals the mean of the principal variate $X_i$, a desirable property. 

Finding a mean adjustment for $U_i^a$ in the control space such that the unbiasedness is satisfied in the principal space is a non-trivial problem.
Explicitly, we seek a vector $\nu$ such that: 
\begin{equation}\label{eq:NL-MFEnKF-inverse-problem-mean-correction}
    \MeanE{\phi(\hat{U}_i^a)} = \MeanE{\phi(U_i^a - \MeanE{U_i^a} + \nu)}.
\end{equation}
%
The solution to \eqref{eq:NL-MFEnKF-inverse-problem-mean-correction} requires the solution to an expensive inverse problem.
Note that \eqref{eq:NL-MFEnKF-inverse-problem-mean-correction} is equivalent to \eqref{eq:NL-MFEnKF-control-space-mean-correction} under the assumptions of Theorem \ref{thm:non-linear-control-gaussian}, in the limit of ensemble size.

\subsubsection{Kalman approximate mean adjustment}
Instead of assuming that the control and ancillary variates are unbiased, we can consider directly the mean of the control-space total variate:
\begin{equation}
    \MeanE{\theta(Z_i^a)} = \MeanE{\theta(X_i^a)} - \half\, \left(\MeanE{\hat{U}_i^a} - \MeanE{U_i^a}\right),
\end{equation}
defined with the mean values in NL-MFEnKF formulas \eqref{eq:NL-MFEnKF}. The following adjustment to the mean of the ancillary variate:
\begin{equation}
    \MeanE{U_i^a} \leftarrow \MeanE{\theta(Z_i^a)},
\end{equation}
is not unbiased with respect to the control variate in any space, but provides a heuristic approximation of the total variate mean in control space.

\subsection{Telescopic extension}
As in \cite[Section 4.5]{popov2021multifidelity} one can telescopically extend the bi-fidelity NL-MFEnKF algorithm to $\!L+1$ fidelities.
Assume that the nonlinear operator $\phi_\ell$ interpolates from the space of fidelity $\ell$ to the space of fidelity $\ell-1$, where $\phi_1$ interpolates to the principal space.  A telescopic extension of \eqref{eq:NL-MFEnKF-control-variate} is
\begin{equation}
    Z = X - \sum_{\ell = 1}^{\!L} 2^{-\ell} \left(\overline{\phi}(\hat{U}_\ell) - \overline{\phi}(U_\ell)\right),
\end{equation}
where the projection operator at each fidelity is defined as,
\begin{equation}
    \overline{\phi}_\ell = \phi_1 \circ \cdots \circ \phi_\ell,
\end{equation}
projecting from the space of fidelity $\ell$ to the principal space.
The telescopic extension of the NL-MFEnKF is not analyzed further in this work.

\section{Models for numerical experiments}
\label{sec:models}

For numerical experiments we use the Lorenz '96 model, and two surrogate models that approximate it's dynamics:
\begin{enumerate}[topsep=3pt,itemsep=0pt]
\item a principal orthogonal decomposition-based quadratic reduced order model (POD-ROM), and
\item an autoencoder neural network-based reduced order model (NN-ROM). 
\end{enumerate}
For each case we construct reduced order models (ROMs) for reduced dimension sizes of $r= 7$, $14$, $21$, $28$, and $35$.
The Lorenz '96 and the POD-ROM models are implemented in the ODE-test-problems suite~\citep{otp,otpsoft}.

\subsection{Lorenz '96}

The Lorenz '96 model \citep{lorenz1996predictability} is conjured from the PDE \citep{reich2015probabilistic},
\begin{equation}
    \frac{\mathrm{d} \*y}{\mathrm{d} t} = -\*y\*y_x - \*y + F,
\end{equation}
where the forcing parameter is set to $F=8$. In the semi-discrete version $\*y \in \mathbb{R}^n$, and the nonlinear term is approximated by a numerically unstable finite difference approximation,
\begin{equation}
    \left[\*y\*y_x\right]_k = \bigl(\widehat{\*I}\*y\bigr)\cdot \bigl(\widehat{\*D}\*y\bigr) = \left([\*y]_{k-1}\right)\cdot\left([\*y]_{k-2} - [\*y]_{k+1}\right),
\end{equation}
where $\widehat{\*I}$ is a (linear) shift operator, and the linear operator $\widehat{\*D}$ is a first order approximation to the first spatial derivative. The canonical $n=40$ variable discretization with cyclic boundary conditions is used. The canonical fourth order Runge-Kutta method is used to discretize the time dimension.

For the given discrete formulation of the Lorenz '96 system, $14$ represents the number of non-negative Lyapunov exponents, $28$ represents the rounded-up Kaplan-Yorke dimension of $27.1$, and $35$ represents an approximation of the intrinsic dimension of the system (calculated by the method provided by \cite{bahadur2019dimension}). To the authors' knowledge, the inertial manifold of the system, if it exists, is not known. The relatively high ratio between the intrinsic dimension of the system and the spatial dimension of the system makes constructing a reduced order model particularly challenging.

\subsubsection{Data for constructing reduced-order models.}
The data to construct the reduced order models is taken to be $T=5000$ state snapshots from a representative model run. The snapshots are spaced $36$ time units apart, equivalent to six months in model time.

\subsection{Proper Orthogonal decomposition ROM} 

Using the method of snapshots \citep{sirovich1987turbulence1}, we construct optimal linear operators, $\*\Phi^T = \*\Theta \in \mathbb{R}^{r\times n}$, such that the projection captures the dominant orthogonal modes of the system dynamics.
The reduced order model approximation with linear projection and interpolation operators \eqref{eq:linear-reduced-order-model} is quadratic (similar to \cite{mou2021reduced,san2015stabilized})
\begin{equation}
    \frac{\mathrm{d}\*u}{\mathrm{d} t} = \*a + \*B \*u + \*u^T \!C \*u,
\end{equation}
where the corresponding vector $\*a$, matrix $\*B$, and 3-tensor $\!C$ are defined by:
\begin{subequations}
    \begin{align}
        \*a &= F\*\Theta\*1_{n},\\
        \*B &= -\*\Theta\*\Phi,\\
        [\!C]_{pqr} &= -\bigl(\widehat{\*I}\*\Phi_{:,p}\bigr)^T\, \bigl(\widehat{\*D}\*\Phi_{:,q}\bigr)\,\*\Phi_{:,r}.
    \end{align}
\end{subequations}

\subsection{Autoencoder based ROM}

We now discuss building the neural network-based reduced order model (NN-ROM). Given the principal space variable $X$, consider an encoder that computes its projection $\hat{U}$ onto the control space \eqref{eq:projection}, and a decoder that computes the reconstruction $\widetilde{X}$ \eqref{eq:reconstruction}.

Canonical autoencoders simply aim to minimize the reconstruction error:
\begin{equation}\label{eq:autoencoder-reconstruction-error}
    X \approx \tilde{X},
\end{equation}
which attempts to capture the dominant modes of the intrinsic manifold of the data. However, for our purposes, we seek to preserve two additional properties.

The first property that is especially important to our application is the stable preservation of the latent space through multiple projections and interpolations:
\begin{equation}\label{eq:latent-reconstruction}
    \theta(\phi(\hat{U}) ) \approx \hat{U},
\end{equation}
which we will call the {\it left-inverse property}.
For POD, this property is automatically preserved by construction and the linearity of the methods. For non-linear operators, the authors have not explicitly seen this property preserved, however, as the MFEnKF requires sequential applications of projections and interpolations, we believe that for the use-case outlined in this paper, the property is especially important.

The autoencoder structure automatically induces a reduced order model $\!M^{AE}$, corresponding to the equation \eqref{eq:reduced-order-model}, that captures the dynamics of the system in the  $r$-dimensional control space. The second property is that the $r$-dimensional ROM dynamics follows accurately the $n$-dimensional dynamics of the full order model.

\begin{remark}
    Note that unlike the POD-ROM whose linear structure induces a purely $r$-dimensional initial value problem, the NN-ROM \eqref{eq:reduced-order-model} still involves $n$-dimensional function evaluations. In a practical method it would be necessary to reduce the internal dimension of the ROM, however that is significantly outside the scope of this paper.
\end{remark}

We wish to construct a proof-of-concept neural network surrogate. We seek to ensure that the induced dynamics \eqref{eq:reduced-order-model} makes accurate predictions, by not only capturing the intrinsic manifold of the data, but also attempting to capture the inertial manifold of the system.  Explicitly, we wish to ensure that
\begin{equation}\label{eq:model-reconstruction}
   \norm{ \!M(X) - \phi(\!M^{AE}(\theta(X))) }
\end{equation}
is minimized. In this sense \eqref{eq:reduced-order-model} would represent an approximation of the dynamics along a submanifold of the inertial manifold.
In practice we replace \eqref{eq:model-reconstruction} by the difference between a short trajectory in the full space started from a certain initial value, and the corresponding short trajectory in the latent space started form the projected initial value. 

Combining the canonical autoencoder reconstruction error term~\eqref{eq:autoencoder-reconstruction-error}, the latent space reconstruction error term (preserving property~\eqref{eq:latent-reconstruction}), and the temporal reconstruction error~\eqref{eq:model-reconstruction} (computed along $K$ small steps), we arrive at the following loss function for each snapshot:
\begin{multline}
    \ell_j(X_j) = \frac{1}{n}\norm{X_j - \phi(\theta(X_j))}_2^2 + \frac{\lambda_1}{r}\norm{\theta(X_j) - \theta(\phi(\theta(X_j)))}_2^2\\ + \sum_{k=1}^K \frac{\lambda_2}{n}\norm{\!M_{t_j, (t_j + k\Delta t)}( X_j ) - \phi(\!M^{AE}_{t_j, (t_j + k\Delta t)}( \theta(X_j) ) ) }_2^2,
\end{multline}
where the hyper-parameters $\lambda_1$ and $\lambda_2$ represent the inverse relative variance of the mismatch.

The full loss function, combining the cost functions for all $T$ snapshots,
\begin{equation}
    L(X) = \sum_{j = 1}^T \ell_j(X_j),
\end{equation}
can be minimized through typical stochastic optimization methods.

\subsubsection{Lorenz '96 surrogate}
Similar to the POD model, we construct $r$-dimensional NN-based surrogate ROMs.
To this end, we use one hidden layer networks with the $\tanh$ activation function for the encoder \eqref{eq:projection} and decoder \eqref{eq:reconstruction}:
\begin{align}
    \theta(X) &= \begin{multlined}[t]\*W^\theta_2\tanh(\*W^\theta_1 X + \*b^\theta_1) + \*b^\theta_2,\\ \*W^\theta_1\in\mathbb{R}^{h\times n}, \*W^\theta_2\in\mathbb{R}^{r\times h}, \*b^\theta_1\in\mathbb{R}^{h}, \*b^\theta_2\in\mathbb{R}^{r},
        \end{multlined}\\
    \phi(U) &= \begin{multlined}[t]\*W^\phi_2\tanh(\*W^\phi_1 U + \*b^\phi_1) + \*b^\phi_2,\\ \*W^\phi_1\in\mathbb{R}^{h \times r}, \*W^\phi_2\in\mathbb{R}^{n\times h}, \*b^\phi_1\in\mathbb{R}^{h}, \*b^\phi_2\in\mathbb{R}^{n},
    \end{multlined}
\end{align}
where the hidden layer dimension on both the encoder and decoder is set to $h=200$ in order to approximate well most non-linear projection and interpolation functions.

The hyperparameter values used in the numerical experiments are $\lambda_1 = 10^3$, $\lambda_2 = 1$, and $K = 5$. We employ the ADAM \citep{kingma2014adam} optimization procedure to train the NN and produce the various ROMs needed. Gradients of the loss function are computed through automatic differentiation.

\section{Numerical experiments}
\label{sec:numerical-experiments}

The numerical experiments compare the following four methodologies : 
\begin{enumerate}[topsep=3pt,itemsep=0pt]
\item Standard EnKF with the Lorenz '96 full order model; 
\item MFEnKF with the POD surrogate model, an approach named  MFEnKF(POD); 
\item NL-MFEnKF with the autoencoder surrogate model, named NL-MFEnKF(NN); and 
\item MFEnKF with the autoencoder surrogate model, named MFEnKF(NN). 
\end{enumerate}
Since MFEnKF does not support non-linear projections and interpolations, in MFEnKF(NN) the ensembles are interpolated into the principal space, and assimilated under the assumption that $\*\Theta = \*\Phi = \*I$.

To measure the accuracy of the analysis mean with respect to the truth we use the spatio-temporal root mean square error (RMSE): 
\begin{equation}
    \text{RMSE} = \sqrt{\sum_{i=1}^N \sum_{k=1}^n \left(\left[\MeanE{\xa_i}\right]_k - \left[X^t_i\right]_k\right)^2},
\end{equation}
where $N$ is the number of data assimilation steps for which we wish to measure the error.

For sequential data assimilation experiments we observe all $40$ variables of the Lorenz '96 system, with an observation error covariance of $\Cov{\eta_i}{\eta_i} = \*I$. Observations are performed every $0.05$ time units corresponding to 6 hours in model time. We run $20$ independent realizations (independent ensemble initializations) for $1100$ time steps, but discard the first $100$ steps for spinup.

\subsection{Accuracy of ROM models}

\begin{table}[t]
    \centering
    \begin{tabular}{|c|c|c|}
        \hline
        $r$ & \textbf{POD-ROM} & \textbf{NN-ROM} \\ \hline
        $7$ & $0.52552$ & $0.56753$ \\
        $14$ & $0.70200$ & $0.76303$ \\
        $21$ & $0.82222$ & $0.88997$ \\
        $28$ & $0.90161$ & $0.95494$ \\
        $35$ & $0.96251$ & $0.98483$ \\ \hline
    \end{tabular}
    \caption{Relative kinetic energies preserved by the reconstructions of the POD-ROM and the NN-ROM solutions of the Lorenz '96 system. Various reduced-order model dimensions $r$ are considered.}
    \label{tab:ROM-kinetic-energy}
\end{table}

Our first experiment is concerned with the preservation of energy by our ROMs, and seeks to compare the accuracy of NN-ROM against that of POD-ROM.
For the Lorenz '96 model, we use the following equation \citep{karimi2010extensive} to model the spatio-temporal kinetic energy,
\begin{equation}
    \text{KE} = \sum_{i=1}^{T} \sum_{k=1}^{n} \left([\*y_i]_k\right)^2,
\end{equation}
where $T$ is the number of temporal snapshots that were used to construct the ROMs.
Table \ref{tab:ROM-kinetic-energy} shows the relative kinetic energies  of the POD-ROM and the NN-ROM reconstructed solutions \eqref{eq:interpolation} (the energies of the reconstructed ROM solutions are by the kinetic energy of the full order solution).

The results lead to several observations. First, the NN-ROM always preserves more energy than the POD-ROM. We have achieved our goal to build an NN-ROM that is more accurate than the POD-ROM. Second, the NN-ROMs with dimensions $r=21$ and $r=28$ preserve almost as much energy as the POD-ROMs with dimension $r=28$ and $r=35$, respectively. Intuitively this tells us that they should be just as accurate.

\subsection{Impact of ROM dimension}

\begin{figure}[t]
    \centering
    \includegraphics[width=0.75\linewidth]{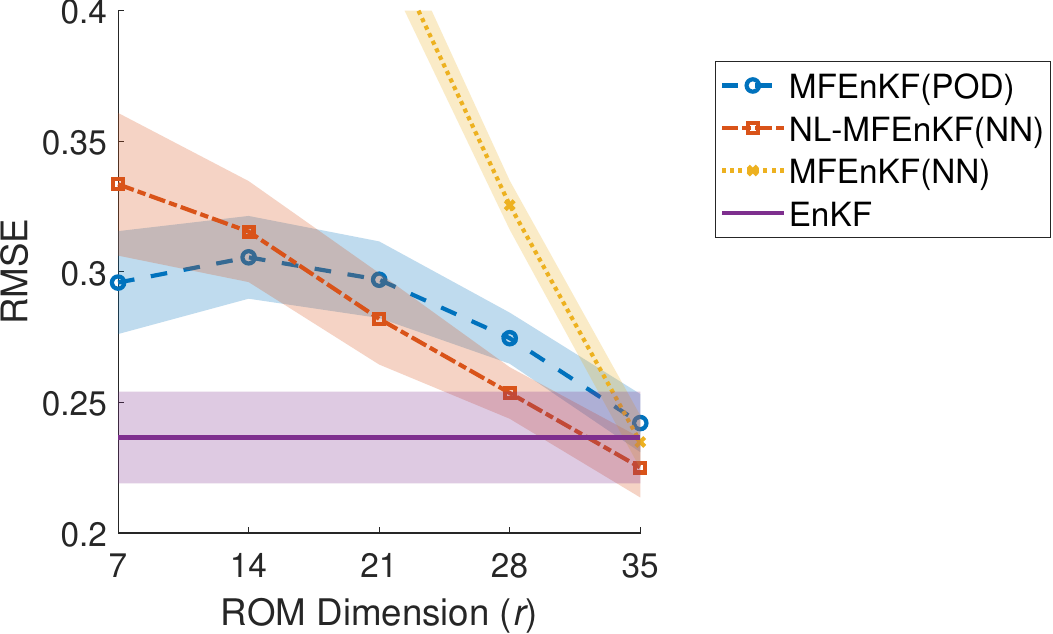}
    \caption{Analysis RMSE versus ROM dimension for three multifidelity data assimilation algorithms and the classical EnKF. Ensemble sizes are $N_X=32$ and $N_U = r-3$. Error bars show two standard deviations. The inflation factor for the surrogate ROMs is fixed at $\alpha_U=1.01$, the inflation of $\alpha = 1.07$ is used for the EnKF, and $\alpha_X=1.05$ is used for all other algorithms.}
    \label{fig:rom-dimension-experiment}
\end{figure}

The second set of experiments seeks to learn how the ROM dimension affects the accuracy of the various multifidelity data assimilation algorithms.

We take the principal ensemble size to be $N_X = 32$, and the surrogate ensemble sizes equal to $N_U = r - 3$, in order to always work in the undersampled regime. All multifidelity algorithms are run with inflation factors $\alpha_X = 1.05$ and $\alpha_U = 1.01$. The traditional EnKF using the full order model is run with an ensemble size of $N=N_X$ and an inflation factor $\alpha = 1.06$ to ensure stability.

The results are shown in Figure \ref{fig:rom-dimension-experiment}. For the `interesting' dimensions $r=28$, $r=35$, and an underrepresented $r=21$, the NL-MFEnKF(NN) performs significantly better than the MFEnKF(POD).
For a severely underrepresented ROM dimension of $r=7$ and $r=14$ the MFEnKF(POD) outperforms or ties the NL-MFEnKF(NN). The authors suspect that this is either due to the NN-ROMs sharing the same hyperparameters, meaning that if the hyperparameters were tailored to each $r$ dimensional NN-ROM individually, the authors suspect that it would beat the POD-ROM in all cases. 

Of note is that excluding the case of $r=35$, the MFEnKF(NN), using the standard MFEnKF method in the principal space, is the least accurate among all algorithms, indicating that the non-linear method presented in this paper is truly needed for models involving non-linear model reduction.

We note that for $r=35$, the suspected intrinsic dimension of the data, the NL-MFEnKF(NN) outperforms the EnKF, both in terms of RMSE and variability within runs. This is additionally strengthened by the results of the MFEnKF(NN) assimilated in the principal space, as it implies that there is little-to-no loss of information in the projection into the control space. 

\subsection{Ensemble Size and Inflation} 

\begin{figure}[t]
    \centering
    \includegraphics[width=0.98\linewidth]{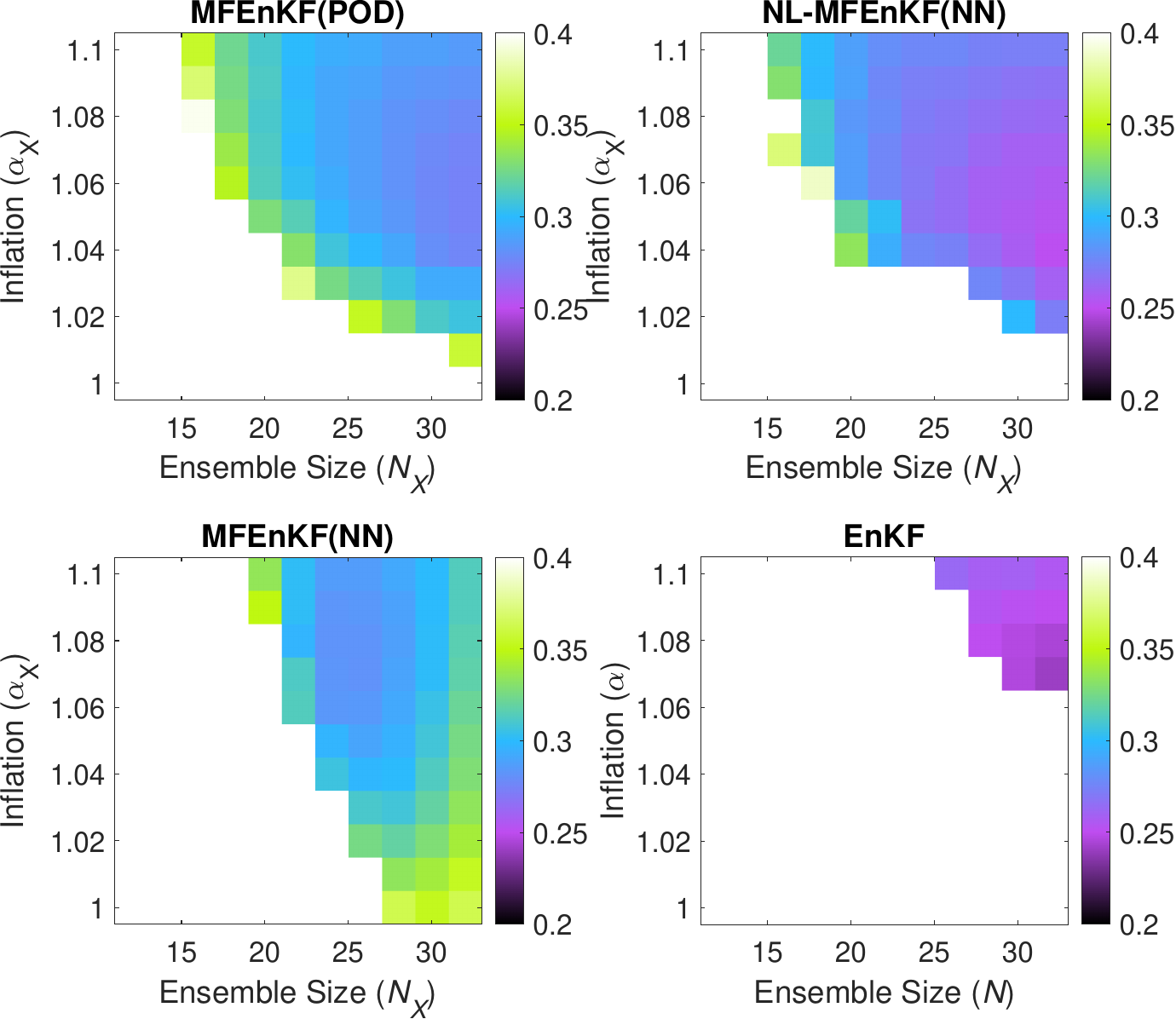}
    \caption{Analysis RMSE for four data assimilation algorithms and various values of the principal ensemble size $N_X$ and inflation factor $\alpha_X$. The surrogate ROM size is fixed to $r=28$ with ensemble size of $N_U=25$ and inflation of $\alpha_U=1.01$. The top left panel represents the MFEnKF with a POD-ROM surrogate, the top right panel represents the NL-MFEnKF with a NN-ROM, the bottom left panel represents the MFEnKF with a NN-ROM surrogate assimilated in the principal space, and the bottom right panel represents the classical EnKF.}
    \label{fig:full-experiment}
\end{figure}

Our last set of experiments focuses on the particular ROM dimension $r=28$, as we believe that it is representative of an operationally viable dimension reduction,  covering the dimensionality of the global attractor.

For each of the four algorithms we vary the principal ensemble size $N_X = N$, and the principal inflation factor $\alpha_X = \alpha$. As before, we set the control ensemble size to $N_U = r-3 = 25$ and the control-space inflation factor to $\alpha_U = 1.01$.

Figure \ref{fig:full-experiment} shows the spatio-temporal RMSE for various choices of ensemble sizes and inflation factors. The results show compelling evidence that NL-MFEnKF(NN) is competitive when compared to MFEnKF(POD); the two methods have similar stability properties for a wide range of principal ensemble size $N_X$ and principal inflation $\alpha_X$, but NL-MFEnKF(NN) yields smaller analysis errors for almost all scenarios for which the two methods are stable.

For a few points with low values of principal inflation $\alpha_X$, the NL-MFEnK(NN) is not as stable as the MFEnKF(POD). This could be due to either an instability in the NN-ROM itself,  in the NL-MFEnKF itself, or in the projection and interpolation operators.

An interesting observation is that the MFEnKF(NN), which is assimilated naively in the principal space, becomes less stable for larger ensemble sizes $N_X$. One possible explanation for this is that the ensemble mean estimates become more accurate, thus the bias between the ancillary and control variates is amplified in \eqref{eq:control-variate}, and more error is introduced from the surrogate model. This is in contrast to most other ensemble based methods, including all others in this paper, whose error is lowered by increasing ensemble size.

\section{Conclusions}
\label{sec:conclusions}

This work extends the multifidelity ensemble Kalman filter framework to work with nonlinear surrogate models. The control variate framework is generalized to incorporate optimal non-linear projection and interpolation operators implemented using autoencoders.

The results obtained in this paper indicate that reduced order models based on non-linear projections that fully capture the intrinsic dimension of the data provide excellent surrogates for use in multifidelity sequential data assimilation. Moreover, nonlinear generalizations of the control variate framework result in small approximation errors, and thus the assimilation can be carried out efficiently in the space of a nonlinear reduced model.

Future directions to be pursued include more robust constructions of projection and interpolation operators, the use of recurrent neural network models, and applications to operational test problems.

\acknowledgements

This work was supported by awards NSF ACI--1709727, NSF CDS\&E-MSS--1953113, DOE ASCR DE--SC0021313,
and by the Computational Science Laboratory at Virginia Tech.

\bibliographystyle{Bibliography_Style}
\bibliography{Bib/biblio,Bib/traian,Bib/sandu,Bib/data_assim_multilevel,Bib/data_assim_kalman}

\end{document}